\documentclass[twoside,11pt]{article}

\usepackage{amsmath,amssymb,amsfonts,amsthm,amsbsy}
\usepackage[utf8]{inputenc}
\usepackage{upgreek}
\usepackage{hyperref}
\usepackage{mathrsfs}
\usepackage{bm}
\usepackage{enumerate}
\usepackage{graphicx}
\usepackage{color}
\usepackage{sectsty}
\usepackage{tocloft}
\usepackage[margin=1.6in]{geometry}

\theoremstyle{plain}
\newtheorem{theorem}{Theorem}[section]

\newtheorem{corollary}[theorem]{Corollary}
\newtheorem{proposition}[theorem]{Proposition}

\theoremstyle{definition}
\newtheorem{definition}{Definition}
\newtheorem{example}{Example}

\theoremstyle{remark}
\newtheorem{remark}{Remark}

\newcommand{\va}{\rightarrow}

\newcommand{\df}{\mathrm{d}}

\newcommand{\al}{\alpha}

\newcommand{\oo}{\Omega}

\newcommand{\ac}{\mathcal{A}}

\newcommand{\cinfm}{\mathcal{C}^\infty (M)}
\newcommand{\der}{\mathrm{Der}}
\newcommand{\er}{\mathrm{End}_{\mathbb{R}}}
\newcommand{\ek}{\mathrm{End}_{\mathbb{K}}}
\newcommand{\om}{\Omega (M)}
\newcommand{\fvv}{\Omega (M;E)}

\newcommand{\hk}{\mathrm{Hom}_{\mathbb{K}}}
\newcommand{\se}{\Upgamma (\Uplambda}
\newcommand{\ges}{\Upgamma (E)}
\newcommand{\ma}{(M,\mathcal{A})}
\newcommand{\str}{\mathcal{S}\mathrm{tr}}

\newcommand{\R}{\mathbb{R}}
\newcommand{\Z}{\mathbb{Z}}

\newcommand{\N}{\mathbb{N}}
\newcommand{\C}{\mathbb{C}}
\newcommand{\K}{\mathbb{K}}

\def\nnabla{\nabla \hskip-2.5mm \nabla}

\partfont{\Large\bfseries\centering}

\begin{document}
\title{Quillen superconnections and connections on supermanifolds}
\author{J. V. Beltr\'an$^1$, J. Monterde$^1$ and J. A. Vallejo$^{2,3}$ \\
{\normalsize $^1$Departament de Matem\'atiques, Universitat de Val\`encia (Spain)}\\
{\normalsize $^2$Departamento de Matem\'aticas, Universidad de Sonora (M\'exico)}\\
{\normalsize $^3$In sabbatical leave. Permanent address: Facultad de Ciencias},\\
{\normalsize Universidad Aut\'onoma de San Luis Potos\'i (M\'exico)}\\
{\footnotesize Email (JAV, corresponding author): \texttt{jvallejo@fc.uaslp.mx}}
}
\date{\today}
\maketitle

\begin{abstract}
Given a supervector bundle $E=E_0\oplus E_1 \to M$, we exhibit a parametrization of Quillen superconnections on $E$ by graded
 connections on the Cartan-Koszul supermanifold $(M,\om )$. The relation between the curvatures of both kind
 of connections, and their associated Chern classes, is discussed in detail. In particular, we find that Chern classes for
 graded vector bundles on split supermanifolds can be computed through the associated Quillen superconnections.
\end{abstract}

\section{Introduction}
Quillen superconnections are ordinary
connections (thought as differential operators
on vector-valued forms), but defined on a vector bundle which carries a $\mathbb{Z}_2 -$grading,
$E=E_0\oplus E_1$, and having odd total degree. They have been widely used both in Physics and Mathematics, ever since their appearance in
\cite{Qui} to construct a representative for the Chern character of the index of a family of elliptic operators.

Ne'eman and Sternberg \cite{NS1,NS2} gave a formulation of a Yang-Mills theory based on a
superconnection (rather than conventional connections), including the scalar Higgs field as one of its
components. That idea was very interesting, as one of the main problems of Yang-Mills theories
is precisely the inclusion of the Higgs as one of the basic bosonic fields. However, the model obtained
this way does not include a quadratic term to provide spontaneous symmetry breaking, so it contains
unwanted massless fields. Further steps to remedy this state of affairs have been given in \cite{Lee1,Roe} (for an alternative
treatment within non-commutative geometry, see \cite{Lee2}).

A second source of interest in superconnections is string theory. $D-$branes, joint with fundamental string
states and ordinary field theoretical solitons, form duality multiplets in the context of duality conjectures that explain how the
different existing superstring theories are perturbation expansions of a single theory, called $M-$theory. These $D-$branes can be
understood as topological defects in the worldvolumes of higher dimensional unstable systems of branes, such as brane-antibrane pairs.
This unstability shows itself in the presence of tachyons, which can not removed by the usual Gliozzi-Scherk-Olive (or GSO) projection.
Indeed, it has been suggested that these tachyon fields could be interpreted as Higgs fields, and that they could be incorporated,
again, within the components of certain superconnections, constructed from the supervector bundles determined by the Chan-Paton
gauge bundles of the brane-antibrane system. The Chern character of these bundles, can be used to relate their cohomology to the
$K-$theory group, allowing a topological classification of $D-$brane charges, see \cite{Wit}.
We strongly recommend \cite{OS,Sza} for a detailed account of further developments based on these ideas.

The $D-$branes mentioned above are associated to $(p+1)-$submanifolds of a $10-$dimensional spacetime with $0\leq p\leq 9$.
The question arises of how to translate these ideas to the general case of superbranes moving on a supermanifold, as Quillen
superconnections are defined over ordinary (commutative) manifolds. Rather than trying to extend the notion of Quillen superconnection
to this setting, we show here that they can be put on correspondence with Koszul connections on a certain class of supermanifolds. This class
can be thought as those split supermanifolds endowed with an odd-degree differential (a homological vector field), so they can be seen also as
Gerstenhaber algebras associated to Lie algebroids (see \cite{Vain}).
But, for the sake of simplicity, in this paper we will restrict ourselves to the case of the well-known Cartan-Koszul supermanifold
$(M,\Omega (M))$, whose sheaf of superfunctions is simply the sheaf of (ordinary) differential forms on a manifold, so the homological vector field is simply the exterior differential  (notice that these
split supermanifolds are examples of differential graded manifolds \cite{Vor,Roy}, and they carry an additional $\mathbb{Z}-$grading,
known in Physics as the \emph{ghost number} grading). We hope that, in this way, the geometric meaning of the constructions presented here
will be clearer.

The contents of the paper are as follows. Section \ref{sec2} serves the dual purpose of setting up the main notations
and making the paper self-contained (perhaps at the risk of irritating the expert reader).
Sections \ref{sec3}, \ref{graded-con}, \ref{carkos} review the main properties of Quillen superconnections and Koszul graded connections on supermanifolds,
paying particular attention to the construction of \emph{graded vector bundles} (over a supermanifold) and their relation to
\emph{supervector bundles} (over the base of the supermanifold---an ordinary manifold). The
remaining sections develop a parametrization of Quillen superconnections by means of graded connections on
$(M,\Omega (M))$, and to study the relationship between their associated Chern classes.

\section{Differential operators on graded modules}\label{sec2}

Throughout this paper, $\K$ will denote either $\R$ or $\C$.
Let $G$ be a commutative group endowed with a morphism $\epsilon :G\va \Z/2\Z$, and let $A=\oplus_{g\in G}A^g$ be a $G-$graded
commutative $\K -$algebra with identity element $\mathbf{1}_A$. Elements $a\in A^g$ will be called homogeneous of degree $|a|=g$.
The commutativity in $A$ means $ab=(-1)^{\epsilon (|a|)\epsilon (|b|)}ba$, for all homogeneous $a,b\in A$.
But in this paper $G$ will be either
$\{ 0\}$, $\Z$ or $\Z_2$, so we will simplify this notation by writing $ab=(-1)^{ab}ba$. Note that any $\Z-$grading induces a
corresponding $\Z_2-$grading by collecting even and odd elements. In this case, $A=A^0 \oplus A^1$ is called a superalgebra (and
the $\Z-$modules over it are called supermodules).

Suppose $A$ is simply a $G-$graded $\K-$algebra (not necessarily commutative). It can be viewed as a Lie superalgebra when
endowed with the supercommutator:
\begin{equation}\label{scommutator}
[a,b]=ab-(-1)^{ab}ba,
\end{equation}
for all homogeneous $a,b\in A$, extended to all of $A$ by $\K-$bilinearity. Notice that a $G-$graded algebra $A$ is commutative
precisely when the supercommutator vanishes identically. The properties of the supercommutator \eqref{scommutator} defining the
Lie superalgebra structure are, besides its $\K-$bilinearity (for homogeneous $a,b,c\in A$):
\begin{enumerate}[(a)]\itemsep0.5em
\item Graded skew-symmetry: $[a,b]=-(-1)^{ab}[b,a]$,
\item Graded Jacobi identity: $[a,[b,c]]=[[a,b],c]+(-1)^{ab}[b,[a,c]]$.
\end{enumerate}

If $N,P$ are finitely-generated $A-$modules, the set $\hk (N,P)$ naturally inherits a $G-$graded $A-$module structure. Thus, given a
$\phi\in\hk (N,P)$ and an element $v\in N$, we will have $|\phi (v)|=|\phi|+|v|$. In particular, for each $k\in\N$ we have $G-$graded $A-$modules
$\oo^k (N)\subset\hk (N\times \stackrel{k)}{\ldots}\times N,A)$, and $\oo^k (N;P)\subset\hk (N\times \stackrel{k)}{\ldots}\times N,P)$, whose elements are
the alternate ones. Defining $\oo^0 (N):=A$, $\oo^0 (N;P);=P$, and $\oo^{-k}(N):=\{0\}=:\oo^{-k}(N;P)$ for $k\in\N$ or $-k>\mathrm{rank} N$, we will write
$$
\oo (N)=\bigoplus_{i\in\Z}\oo^i (N)\mbox{ and }\oo (N;P)=\bigoplus_{j\in\Z}\oo^j (N;P).
$$
Thus, $\oo (N;P)$ not only has a $G-$grading, but also a $\Z-$grading. The \emph{total degree} of an element $\phi\in\oo^k (N;P)$ is the pair $(k,|\phi|)\in\Z\times G$. When necessary, we will indicate this total degree by the notation $\phi\in\oo^{(k,|\phi|)} (N;P)$.
\begin{example}\label{ex1}
Let $\pi:E\va M$ be a smooth real vector bundle over the smooth manifold $M$, whose dual bundle will be denoted $E^*$.
The algebra of smooth functions $\cinfm$ is commutative in the usual sense ($G=\{0\}$), the space of smooth sections $\Upgamma E$ is a $\cinfm-$module,
and the exterior bundle $\se E)$ is both a $\cinfm-$module and a $\Z-$graded commutative algebra, $\se E)=\oplus_{i\in\Z}\se^i E)$ (with $\se^0 E)=\cinfm$ and
$\se^{-i} E)=\{0\}$ if $i\in\N$).\\
Taking as the algebra $A=\cinfm$, we can consider the $\cinfm-$module $\mathcal{X}(M) =\der(\cinfm)$ of vector fields on $M$. In this context, the notation
$\oo^k (M;F)$ means that $F$ is another smooth vector bundle over $M$, so its elements (the $F-$valued differential forms on the manifold $M$) are of
the form $\phi:\mathcal{X}(M)\times \cdots\times\mathcal{X}(M)\va F$, $\cinfm-$linear and alternate.\\
If we consider $A=\se E)$, an important $\Z-$graded $\se E)-$module is the one of graded endomorphisms, $\er (\se E))=\oplus_{i\in\Z}\er^i (\se E))$, as well as the
$\se E)-$modules that can be formed by taking tensor products with other vector bundles over $M$.
\end{example}
\begin{remark}
Note that any $a\in A^g$ can be viewed as an element $a\in \ek^g (P)$ for any $G-$graded $A-$module $P$, acting as $a(v)=a\cdot v$, $v\in P$.
\end{remark}

To define differential operators, we will need some facts about the algebraic structure of $\hk (N,P)$, where $N,P$ are $G-$graded $A-$modules,
with $A=\oplus_{g\in G}A^g$.
First, note that if $D\in \hk (N,P)$, and $a\in A$, then we can define $[D,a]\in \hk (N,P)$ by putting, for all $v\in N$:
\begin{equation}\label{eq1}
[D,a](v):=D(av)-(-1)^{aD}aD(v)\in P.
\end{equation}
It is immediate to check that this bracket satisfies:
\begin{equation}\label{eq1b}
[[D,D'],a]=[D,[D',a]]-(-1)^{DD'}[D',[D,a]],
\end{equation}
for any pair $D,D'\in \hk (N,P)$, and also:
\begin{equation}\label{deriv}
[D,ab]=[D,a]b+(-1)^{aD}a[D,b],
\end{equation}
for any $a,b\in A$.
\begin{definition}\label{def1}
Let $A$ be a $G-$graded commutative $\K-$algebra, and let $N,P$ be $G-$graded $A-$modules. A morphism $\Updelta\in\hk (N,P)$ is called a ($G-$graded)
differential operator of order $\leq k$ (with $k\in \N\cup\{0\}$) if, for all $a_0,a_1,\ldots,a_k\in A$, the following holds:
$$
[[\cdots [[\Updelta,a_0],a_1],\cdots ],a_k]=0.
$$
\end{definition}
The set of all such operators will be denoted $\mathcal{D}_k (N;P)$ (or $\mathcal{D}_k (N)$ in the case $P=N$). It is clear that $\mathcal{D}_k (N;P)$ is an
$A-$module, as for any $a,b\in A$, $\Updelta\in\mathcal{D}_k (N;P)$:
$$
[a\Updelta,b]=a\Updelta (b)-(-1)^{(a+\Updelta)b}ba\Updelta=a\Updelta(b)-(-1)^{(a+\Updelta)b+ab}ab\Updelta=a[\Updelta,b].
$$
Thus, we have a $\Z-$graded $A-$module $\mathcal{D} (N;P):=\oplus_{k\in\Z}\mathcal{D}_k (N;P)$, with the usual conventions for $k<0$. To account for the
$G-$grading, let us define
$$
\mathcal{D}^p_k (N;P):=\mathcal{D}_k (N;P)\bigcap \hk^p (N;P).
$$
\begin{example}\label{ex2}
Let $E$ be a real vector bundle as in Example \ref{ex1}. A linear (Koszul) connection $\nabla$ on $E$ is an $\R-$bilinear mapping $\nabla:
\mathcal{X}(M)\times\Upgamma E\va\Upgamma E$
(with $\nabla(X,\sigma)$ denoted $\nabla_X\sigma$) such that $\nabla_X(f\sigma)=\df f(X)\sigma +f\nabla_X\sigma$, for all $f\in\cinfm$. This can
be written as
$[\nabla,f](X)(\sigma)=\df f(X)\cdot\sigma$, or $[\nabla,f]=\df f$ as an $E-$valued $1-$form. The mapping $\nabla$ can
be extended to an operator
$\df^\nabla:\oo^p(M;E)\va\oo^{p+1}(M;E)$ by defining (for $\phi\in\oo^p (M;E)$, $X_0,\ldots,X_p\in\mathcal{X}(M)$):
$$
(\df^\nabla\phi)(X_0,\ldots,X_p)=\sum_{i=0}^p s_i(\nabla_{X_i}\phi)(X_0,\stackrel{\hat{i}}{\ldots},X_p)+\sum_{i<j}s_{i+j}\phi([X_i,X_j],X_0,
\stackrel{\hat{i}\hat{j}}{\ldots},X_p),
$$
where, as usual, a hat over an index denotes that the corresponding term is omitted, and we have used the shorthand $s_j=(-1)^j$. Now, from
$[\nabla,f]=\df f$ and the property \eqref{deriv} of the commutator \eqref{eq1}, we get that $\df^\nabla$ must be a first-order differential
operator on the $\om-$module $\fvv$: $[[\df^\nabla,\al],\beta]=0$, for all $\al,\beta\in\om$
(just write $\al=f_I\df g^I$ for some multi-index $I=(i_1,\ldots,i_{|\al|})$). Thus, we can characterize linear connections $\nabla$ on $E$ as
the restrictions to $\oo^1(M;E)$ of first-order differential operators on $\fvv$ of degree $1$, $D\in\mathcal{D}^1_1 (\fvv)$, such that
$[D,f]=\df f$, for all $f\in\cinfm$. So, $D=\df^\nabla$.
\end{example}
A case of particular interest is obtained by taking $N=A$ in Definition \ref{def1}.
\begin{definition}\label{def2}
Let $A$ be a $G-$graded commutative $\K-$algebra, $P$ a $G-$graded $A-$module. A morphism $D\in\hk (A;P)$ is called a ($G-$graded) $P-$derivation
on $A$, of $G-$degree $k$, if for all $a,b\in A$:
$$
D(ab)=D(a)b+(-1)^{ak}aD(b)\in P.
$$
If, moreover, $D$ satisfies $[D,a]=0$, for all $a\in A$, we say that $D$ is an \emph{algebraic derivation}.
\end{definition}
The set of all $P-$derivations on $A$ is denoted $\der (A;P)=\oplus_{k\in G}\der^k (A;P)$. If $\Updelta\in\mathcal{D}_1 (A;P)$, then $\Updelta-\Updelta(\mathbf{1}_A)\in\der (A;P)$,
as it is readily checked. Thus, we have the decomposition of the space of first-order differential operators in this case:
$\mathcal{D}_1 (A;P)=\der (A;P)\oplus P$.\\
When $P=A$, we simply write $\der A$, and refer to any $\delta\in\der A$ as a ($G-$graded) derivation of $A$. In particular,
$\mathcal{D}_1 (A)=\der A\oplus A$.
\begin{remark}\label{rem2}
A case of particular interest appears when $Q$ is any $\K-$module and we form the left $A-$module $A\otimes Q$.
If $\delta\in\der(A\otimes Q)$, and $a\in A$, in general we have $[\delta ,a]\in\mathrm{End}(A\otimes Q)$. But if
it happens that, for each $a\in A$, this morphism is of the form `multiplication by an element of $A$', and we denote this element by
$\overline{\delta}a$ (as it depends on both, $\delta$ and $a$), then $\overline{\delta}:A\to A$ is a derivation, as
a consequence of \eqref{deriv}:
$$
\overline{\delta}(ab)=[\delta,ab]=[\delta,a]b+(-1)^{a\delta}a[\delta,b]=\overline{\delta}(a)b+(-1)^{a\delta}a\overline{\delta}(b) .
$$
In this case, we say that $\overline{\delta}$ is the derivation in $A$ induced by $\delta$.
\end{remark}
\begin{example}\label{ex3}
Let $E$ be a smooth real vector bundle over the manifold $M$.
\begin{enumerate}[(a)]
\item\label{itb} Given a linear connection $\nabla$ on $E$, its extension to a differential operator on $\fvv$ (as in Example
\ref{ex2}) is done in such a way that it acts as a derivation of degree $1$.
\item\label{ita} Given a $P\in\oo^k (M;\ek (E))$, we can define a derivation of degree $k$ (which will be denoted also as $P$) by:
\begin{align*}
&(P\sigma)(X_1,\ldots,X_{p+k})= \\
&\frac{1}{p!k!}\sum_{\pi\in S_{p+k}}\mathrm{sg}(\pi)P(X_{\pi(1)},\ldots,X_{\pi(k)})
(\phi (X_{\pi(k+1)},\ldots,X_{\pi(p+k)})),
\end{align*}
for any vector fields $X_1,\ldots,X_{p+k}$, and $\sigma\in\oo^p (M;E)$ (here $\mathrm{sg}(\pi)$ denotes the signature
of the permutation $\pi$). In particular, if $P\in\oo^1 (M;E)$, then
$P\sigma (X)=P(X)(\sigma)\in\Upgamma (E)$.
These are examples of algebraic derivations.
\item\label{itc} Again, let $\nabla$ be a linear connection on $E$. Given any $a\otimes X\in\se^a E\otimes TM)$, we have an
endomorphism $\nabla_{a\otimes X}\in\ek^a (\se E))$ defined by $\nabla_{a\otimes X}(b):=a\wedge \nabla_X (b)$. If
$K\in\se^a E\otimes TM)$ is an arbitrary element, we define $\nabla_K$ by its linear extension. This is a derivation
on $\se E)$ of $\Z-$degree $a$.
\item\label{itd} Given any $L\in\se E\otimes E^*)$, it can be expressed as a sum of decomposable elements, each of them of the
form $b\otimes\beta$, for some homogeneous section $b\in\se^b E)$, and $\beta\in\Upgamma (E^*)$. For these, we define
$i_{b\otimes\beta}:\se E)\va \se E)$ by putting $i_{b\otimes\beta}(a)=b\wedge i_{\beta}(a)$, and then extend by
linearity. In this way, we get an element of $\der^{b-1}\se E)$.
\end{enumerate}
\end{example}
\begin{remark}\label{michor}
An important consequence can be drawn from the preceding examples. Notice that once
a linear connection $\nabla_0$ on $E$ has been chosen, any derivation $D\in\der^1\fvv$ such that its induced derivation
on $\om$ is $\overline{D}=\df$, decomposes as:
$$
D=\df^{\nabla_0}+P,
$$
for some $P\in\oo^1(M;\ek (E))$ (see 3.8 in \cite{Mic}). In particular, $\df^{\nabla}=\df^{\nabla_0}+P$.
\end{remark}

\section{Quillen superconnections}\label{sec3}

In the examples of the preceding section, which were intended to illustrate the classical notions of
differential geometry from the point of view of differential operators on modules, we chose a smooth
vector bundle $E\va M$, whose sections are a module over the commutative algebra $\cinfm$. In this way,
we obtained a $\Z-$graded algebra $\se E)$, but forgot the $G-$grading (which appears in Definitions
\ref{def1} and \ref{def2}) or, more precisely, took it to be trivial. Quillen's definition of
superconnection, takes full advantage of the $G\times\Z-$bigrading by considering a $\K -$supervector bundle
$E=E_0\oplus E_1$ over $M$. But, aside from this fact, the definition is formally the same as the classical one
described in Example \ref{ex2}.\\
In this section we give only the basic facts about Quillen superconnections that will be required to prove our
main result. A detailed study of this notion can be found in the original paper \cite{Qui} and in
\cite{BGV,Sardanas}.

\begin{definition}\label{def3}
A Quillen superconnection (or simply superconnection) on a $\K -$su\-per\-vector bundle over $M$, $E = E_0\oplus E_1$ is a derivation
on $\Omega(M;E)$ (thus, a first-order differential operator), of odd \emph{total} degree, which induces the exterior differential
$\df$ on $\om$, that is:
\begin{equation}\label{def-Quillen}
D(\alpha\wedge\sigma)= {\rm d}\alpha\wedge\sigma +(-1)^{\alpha} \alpha\wedge D\sigma ,
\end{equation}
for $\alpha\in\Omega(M)$ and $\sigma\in\Omega(M;E)$.
\end{definition}
\begin{remark}\label{rem-qui}
This definition is the one given in \cite{Qui}. But note that condition (\ref{def-Quillen}) is equivalent to:
$$
[D,\alpha](\phi) = {\rm d}\alpha\wedge\phi ,
$$
in terms of the graded commutator \eqref{eq1}.
\end{remark}

As in the classical case, the difference between two superconnections is an $\mathrm{End}(E)-$valued form on $M$.
Indeed, it is an algebraic operator in the sense of Definition \ref{def2} as, if
$D$ and $D'$ are superconnections, and $\al\in\om$, then $[D-D',\al ]=\df\al -\df\al =0$. The following result,
proved in \cite{BGV}, uses the identification mentioned in \eqref{ita} of Example \ref{ex3}.
\begin{proposition}
 The space of superconnections on $E$ is an affine space modeled on the vector space:
 $$
\displaystyle \Omega^{-}(M;\ek(E)) := \bigoplus_{i+k\mbox{ odd}} \Omega^{(i,k)}(M;\ek(E)).
 $$
\end{proposition}

Noticing that Definition \ref{def3} involves the \emph{total} degree, we see that the space modeling Quillen superconnections
is built from sections of the following vector bundles:
\begin{equation}\label{modeled-Quillen}
\begin{array}{lcr}
\displaystyle\bigoplus_{k\ge 0}\se^{2k+1}T^\ast M\otimes\ek^0(E)), \\[3mm]
\displaystyle\bigoplus_{k\ge 0}\se^{2k}T^\ast M\otimes\ek^1(E)),
\end{array}
\end{equation}
where $\ek^0(E)=(E^*_0\otimes E_0)\oplus (E^*_1 \otimes E_1)$ and $\ek^1(E)=(E^*_0\otimes E_1)\oplus (E^*_1 \otimes E_0)$.\\
This fact, combined with Remark \ref{michor}, shows that when the vector bundle $E=M\times V$ is trivial, any superconnection
can be written as $D=\df +P$, with $P\in\Omega^{-}(M;\ek(V))$. This is the case for most applications in Physics, see
\cite{MHY,KL,NS1,NS2,Tahi,Yasui}.
\begin{example}\label{example-quillen}
Consider any Koszul connection $\nabla$ on the vector bundle $E\to M$, such that it preserves the $\Z_2-$grading of
$E=E_0\oplus E_1$. Any superconnection $D\in\der\fvv$ of total degree $1$, can be written, following Remark \ref{michor},
as $D=\df^{\nabla}+P$, where $P=A+L\in\oo^0(M;\mathrm{End}^1(E))\oplus \oo^1(M;\mathrm{End}^0(E))$.
That means that $A$ is just an element of $\Upgamma (\mathrm{End}^1(E))$, and, for any vector field
$X\in\mathcal{X}(M)$, we have
$L(X)\in \Upgamma (\mathrm{End}^0(E))$. Thus, Quillen superconnections of total degree $1$, can be parameterized by
four tensor fields: $T_1\in\Upgamma (T^*M\otimes E^*_0\otimes E_0)$ and $T_2\in\Upgamma (T^*M\otimes E^*_1\otimes E_1)$
corresponding to $L$, and $T_3\in\Upgamma (E^*_1\otimes E_0)$, $T_4\in\Upgamma (E^*_0\otimes E_1)$ corresponding
to $A$. This is the setting presented in \cite{Qui}.
\end{example}
\begin{definition}
The curvature of a superconnection $D$ is the element $R^D\in \ek\fvv$ given by:
$$
R^D=D^2=\frac{1}{2}[D,D].
$$
\end{definition}
The curvature $R^D$ is thus a second-order differential operator on $\fvv$, of even total degree. Indeed, as for any $\al\in\om$
(applying \eqref{eq1b}):
$$
\frac{1}{2}[[D,D],\al ]=[D,[D,\al ]]=[D,\df \al ]=\df^2\al =0,
$$
we have that $R^D\in\Omega^+ (M;\mathrm{E})$ is an algebraic derivation in the sense of Definition \ref{def2}.

\section{Connections on supermanifolds}\label{graded-con}
The basic idea underlying the definition of a
supermanifold is the replacement of the commutative sheaf of differentiable
functions $\cinfm$ of a manifold $M$, by another one in which
we can accommodate objects with a $\mathbb{Z}_{2}-$grading. The axiomatic for such
spaces was given by M. Rothstein in \cite{Rot}, here we follow the slightly modified version of \cite{BBHP}.
\begin{definition}
\label{superman}A supermanifold of dimension $(m|n)$ and
basis $(M,\cinfm)$ is given by a usual differential manifold $M$, with
dimension $m$, and a sheaf $\mathcal{A}$ of $\Z_2-$graded commutative $\K-$algebras
(called the structural sheaf) such that:
\begin{enumerate}[(1)]
\item There is an exact sequence of sheaves%
\begin{equation}
0\rightarrow\mathcal{N}\rightarrow\mathcal{A}\overset{\sim}{\rightarrow}\cinfm\rightarrow0, \label{eq1_5}%
\end{equation}
where $\mathcal{N}$ is the sheaf of nilpotents of $\mathcal{A}$ and $\sim$ is
a surjective morphism of sheaves of graded commutative $\K-$algebras (called the
structural morphism).
\item \label{Cond2}$\mathcal{N}/\mathcal{N}^{2}$ is a locally free module with
rank $n$ over $\cinfm=\mathcal{A}/\mathcal{N}$, and $\mathcal{A}$ is
locally isomorphic, as a sheaf of $\Z_2-$graded $\K-$commutative algebras, to the
exterior bundle $\Uplambda_{\cinfm}(\mathcal{N}/\mathcal{N}^{2})$.
\end{enumerate}
\end{definition}
By definition, ${\mathcal A}$ is locally isomorphic to the sheaf of sections of the exterior algebra bundle
$\Uplambda F\to M$, for some vector bundle $F\to M$ (canonically attached to $\mathcal{A}$). When  $\mathcal{A}$ is
globally isomorphic to $\se F)$, the supermanifold is said to be split. In particular, any smooth real
supermanifold
is split, but this is not true for arbitrary holomorphic supermanifolds. Indeed, Koszul (see \cite{Kz2}) has
proven that the existence of splittings of a given supermanifold is related to the existence of certain graded linear
connections on it (see Definition \ref{gra-con} below). We will assume for the
remainder of this paper that supermanifolds $(M,\mathcal{A})$ are split. Indeed, the complex structure will play
no role from now on, so we will assume $\mathbb{K}=\mathbb{R}$ for simplicity (and drop the corresponding
subindexes), unless otherwise explicitly stated.

Supervector fields are defined as (graded) derivations of the structural sheaf $\mathcal{A}=\se E)$. The following result, involving the derivations \eqref{itc} and \eqref{itd} of Example \ref{ex3}, is obtained by adapting
the classical proof of the decomposition of a derivation on $\om$ (see \cite{Kos,MM}), and gives us the
structure of supervector fields.
\begin{theorem}\label{thm1}
Let $\pi:F\va M$ be a smooth vector bundle over the differential manifold $M$. Let $\delta\in\der\se F)$, and let
$\nabla$ be a linear connection on $F^*$. Then, there exist unique tensor fields $K\in\se F\otimes TM)$ and
$L\in\se F\otimes F^*)$, such that
$$
\delta=\nabla_K +i_L.
$$
\end{theorem}
\begin{remark}\label{rem3}
By combining this theorem with Remark \ref{rem2}, we get a corresponding decomposition of the space of
first-order superdifferential operators $\mathcal{D}_1 (\se F))$.
\end{remark}
\begin{example}\label{lieder}
If we consider the cotangent bundle of a manifold, $T^*M\to M$, and a linear connection on $TM$, $\nabla$,
it is immediate to check that the Lie derivative with respect to a vector field $X\in\mathcal{X}(M)$,
$\mathscr{L}_X\in\der\se T^*M)$, admits the decomposition:
$$
\mathscr{L}_X =\nabla_X +i_{\nabla X +T(X,\cdot )},
$$
where $T$ is the torsion of $\nabla$.
\end{example}

Differential superforms can be now defined on a supermanifold $(M,\mathcal{A})$ by (graded) duality;
note that they have a natural structure
of sheaves of \emph{right} $\mathcal{A}-$modules, denoted by $\oo^k(M,\mathcal{A})$, $k\in\Z$ (with
$\oo^{-p}(M,\mathcal{A})=\{ 0\}$ if $p>0$). Thus, for instance, the elements of $\oo^1(M,\mathcal{A})$ are
morphisms of sheaves of $\mathcal{A}-$modules $\der (\mathcal{A})\to\mathcal{A}$. Other notions from the usual
calculus on manifolds admit straightforward generalizations (see \cite{Kos} for details).
\begin{example}\label{lad}
In particular, we have
the graded differential, which we will denote $\bm{\df}$, with bidegree $(1,0)$. If
$\langle \cdot ,\cdot \rangle$ denotes the pairing between supervector fields and superforms, it satisfies:
$$
\langle \delta;\bm{\df}\al \rangle =\delta(\al ),
$$
for any $0-$superform $\al \in \mathcal{A}$, and $\delta\in\der (\mathcal{A})$.
\end{example}

\begin{definition}
A graded vector bundle over $(M,\mathcal{A})$, is a locally free sheaf ${\mathcal L}$ of ${\mathbb Z}_2$-graded
$\mathcal{A}$-modules over $M$, ${\mathcal L}= {\mathcal L}_0\oplus {\mathcal L}_1$.
\end{definition}
\begin{example}\label{stangent}
Given a supermanifold $(M,\mathcal{A})$, the sheaf of derivations $\der(\mathcal{A})$ is locally free and
finitely generated (because of condition \eqref{Cond2} in Definition \ref{superman}), and gives
rise to a graded vector bundle over $(M,\mathcal{A})$, which is called the
supertangent bundle.
\end{example}
The following result, which will be essential in what follows, explains the structure of such bundles. For
its proof, see \cite{JA}.
\begin{theorem}\label{superbundles}
Let  $\mathcal{L}$ be a graded vector bundle over $(M, \mathcal{A})$, then there is a supervector bundle over $M$,
$E= E_0\oplus E_1$, such that ${\mathcal L}$ is isomorphic (as ${\mathcal A}$-supermodules) to
${\mathcal A}\otimes \Gamma(E)$.
\end{theorem}
\begin{example}\label{taniso}
Given any split supermanifold $(M,\se F))$, we have the isomorphism
$\der\se F)\simeq \se F) \otimes \Gamma(TM\oplus F^*)$
(simply choose any connection $\nabla$ on the vector bundle $F\to M$, and make use of Theorem \ref{thm1}).
\end{example}
\begin{example}\label{vvsforms}
We will need also the notion of vector-valued $k-$superforms. These are morphisms of sheaves of
$\ac-$supermodules $S:\der(\ac)\times \stackrel{k)}{\ldots}\times\der(\ac)\to\mathcal{L}$, with
$\mathcal{L}$ a graded vector bundle over $\ma$, and form the sheaf of $\ac-$modules $\oo^k(\ma;\mathcal{L})$.
In particular, $\mathcal{L}-$valued $0-$superforms
are just elements of the sheaf of $\ac-$modules $\mathcal{L}$, and $\mathcal{L}-$valued $1-$superforms
can be identified with $\oo^1\ma\otimes\mathcal{L}$.
\end{example}

Next, we define graded connections, mimicking the notion of classical Koszul connection (see \cite{KM}, \cite{Kz2}).
\begin{definition}\label{gra-con}
Let ${\mathcal L}$ be a graded vector bundle over the supermanifold
$(M,{\mathcal A})$.
A graded connection on ${\mathcal L}$ is an \emph{even} morphism of left ${\mathcal A}-$modules from
$\nnabla :{\rm Der}({\mathcal A})\to \mathrm{End}({\mathcal L})$, whose action is denoted
$\delta\mapsto \nnabla_{\delta}$,
satisfying the Leibniz property:
$$
\nnabla_{\delta}(\al \phi) = \delta(\al)\phi + (-1)^{|\al||\delta|}\al \nnabla_{\delta}\phi ,
$$
where $\delta\in {\rm Der}({\mathcal A})$, $\phi\in {\mathcal L}$  and $\al\in \mathcal{A}$.
In particular, a graded connection on the $\mathcal{A}-$module $\mathcal{L}=\der (\mathcal{A})$ is called
a graded linear conection.\\
The graded curvature of $\nnabla$, $R^{\nnabla}$, is given by:
$$
R^{\nnabla}(\delta,\delta')\phi=
\nnabla_{\delta}\nnabla_{\delta'}\phi-(-1)^{|\delta||\delta'|}\nnabla_{\delta'}\nnabla_{\delta}\phi-\nnabla_{[\delta,\delta']}\phi,
$$
where $\delta,\delta'\in \der({\mathcal A})$ and $\phi\in\mathcal{L}$.
\end{definition}
\begin{remark}\label{leibniz}
The Leibniz property above, can be rewritten as
$$
[\nnabla,\al ]=\bm{\mathrm{d}}\al ,
$$
for any $0-$superform $\al\in \mathcal{A}$, in terms of the graded differential (see Example \ref{lad}) and
the graded commutator \eqref{eq1}. Also, the
graded curvature can be defined as
$$
R^{\nnabla}(\delta,\delta')=[\nnabla_{\delta},\nnabla_{\delta'}]-\nnabla_{[\delta,\delta']},
$$
where now the graded commutators are those of $\mathrm{End}(\mathcal{L})$
and $\mathrm{End}(\mathcal{A})$, respectively, in the first and second terms of the right hand side
(notice that the last bracket restricts to another one on $\der(\mathcal{A})$, that is: the bracket
of two derivations is again a derivation, as can be readily checked).
\end{remark}

\section{The Cartan-Koszul supermanifold}\label{carkos}
From Remarks \ref{rem-qui} and \ref{leibniz}, it is clear that, although their definitions are formally
similar, Quillen superconnections and graded connections are quite different objects: the module considered
in the first case is $\fvv$, endowed with the total grading, while in the second, it is
$\oo (\ma;\mathcal{L})$ with its natural $\Z\times\Z_2 -$bigrading. Of course, a reason for that
difference is that (despite its name), Quillen superconnections are defined without reference to any supermanifold.
However, the two notions can be brought closer if we restrict our attention to graded connections on a
particular supermanifold, as we will now see.

Suppose we are given a supervector bundle over the ordinary manifold $M$, $E = E_0\oplus E_1$.
Consider the split supermanifold whose sheaf of superalgebras is the module of differential
forms on the base manifold, $(M,{\mathcal A}) = (M,\Omega(M))=(M,\se T^\ast M))$. This is called the
Cartan-Koszul (or, simply, Cartan) supermanifold \cite{Uhl}.
Also, consider the graded vector bundle on $(M,\om)$ defined by $E = E_0\oplus E_1$, that is,
$\mathcal{L}=\om\otimes\ges$ (recall Theorem \ref{superbundles}). In this case, a graded
connection $\nnabla$ on $(M,\om)$ gives,
for each $\delta\in\der\om$, an even morphism of $\om-$modules $\nnabla_{\delta}:\om\otimes\ges\to\om\otimes\ges$.
We put on $\ges$ a structure of $\om-$module through the injection of $\cinfm-$modules
$\ges\to\bm{1}\otimes\ges\subset \om\otimes\ges$, where $\bm{1}$ is the constant function on $M$,
$\bm{1}(x)=1\in\R$, for all $x\in M$. Then, if $\alpha\in \Omega(M)$, $\phi\in\ges$, we can define
$\al\phi =\bm{1}\otimes \al\phi=\al\otimes\phi$, and:
\begin{align}\label{eq-delta}
\nnabla_{\delta}(\alpha\phi) &= \nnabla_{\delta}(\alpha\otimes\phi)\nonumber \\
						&= \delta(\alpha)\otimes \phi + (-1)^{|\delta||\alpha|} \alpha\otimes \nnabla_{\delta}\phi \\
						&= \delta(\alpha)\phi + (-1)^{|\delta||\alpha|} \alpha \nnabla_{\delta}\phi .\nonumber
\end{align}
Thus, we can think of a graded connection on $\mathcal{L}=\om\otimes\ges$ (over $(M,\om)$) as
a map $\nnabla:{\rm Der}(\om)\otimes\ges\to\ges$, such that for each $\delta\in\der\om$, it gives
an even morphism of $\om-$modules, $\nnabla_{\delta}:\ges\to\ges$, satisfying
the Leibniz rule:
$\nnabla_{\delta}(\alpha\phi)=\delta(\alpha)\phi + (-1)^{|\delta||\alpha|} \alpha \nnabla_{\delta}\phi$.

The supervector fields on $(M,\om)$ are the derivations of $\om$. Theorem \ref{thm1} in this case
states that, once a linear connection $\nabla$ on $TM$ is fixed, any $\delta\in\der\om$ can be written as
$\delta =\nabla_K +i_L$, for some vector-valued forms $K,L\in\se T^*M\otimes TM)$. By the
linearity properties $\nabla_{\al\otimes X}=\al\wedge\nabla_X$, $i_{\al\otimes X}=\al\wedge i_X$ (where
$\al\in\om$ and $X\in\mathcal{X}(M)$), a basis for the locally-free finitely generated sheaf of $\om-$modules $\der\om$,
consists of derivations $\{ \nabla_X ,i_X \}$, and to give the action of a graded connection $\nnabla$, it
suffices to know $\nnabla_{\nabla_X}\phi$, and $\nnabla_{i_X}\phi$, for any $\phi\in\ges$.
\begin{proposition}\label{tensor-associated-graded-connection}
Let $\nnabla$ be a graded connection on the graded vector bundle $\mathcal{L}=\om\otimes\ges$ over the supermanifold $(M,\Omega(M))$,
$\nabla$ be a connection on $TM$, and $\nabla^E$ be a compatible connection on $E$. Then, there is a tensor
field $K=K_0+K_1\in \Upgamma(T^\ast M\otimes \Uplambda T^\ast M\otimes {\rm End}(E))$, with:
\begin{align*}
\displaystyle K_0:TM&\to& \bigoplus_{i\ge 0}\se^{2i} T^\ast M\otimes {\rm End}^0(E))\oplus \bigoplus_{i\ge 0}\se^{2i+1} T^\ast M\otimes {\rm End}^1(E)),\\[3mm]
\displaystyle K_1:TM&\to& \bigoplus_{i\ge 0}\se^{2i+1} T^\ast M\otimes {\rm End}^0(E))\oplus \bigoplus_{i\ge 0}\se^{2i} T^\ast M\otimes {\rm End}^1(E)),
\end{align*}
and such that:
\begin{equation}\label{def-nnabla}
\begin{cases}
\begin{array}{rcl}
\nnabla_{\nabla_X}\phi  &=& \nabla^E_X\phi + K_0(X;\phi),\\[3mm]
\nnabla_{i_X}\phi  &=& K_1(X;\phi).
\end{array}
\end{cases}
\end{equation}
Moreover the connection is completely determined by these two tensor
fields and the connections $\nabla$, $\nabla^E$.
\end{proposition}
\begin{proof}
It is clear that $K_0$ and $K_1$ are determined by \eqref{def-nnabla}; in turn, once $K_0$, $K_1$,
and $\nabla$, $\nabla^E$ are given, $\nnabla$ is uniquely characterized. Thus, we only need to prove that
$K_0$, $K_1$ are tensor fields. We will show this for $K_0$ (the computations in the case of $K_1$ are
analogous). For any $f\in C^\infty(M)$, $X\in\mathcal{X}(M)$, and $\phi\in\ges$:
$$
\begin{array}{rcl}
K_0(fX;\phi)&=&\nnabla_{\nabla_{fX}}\phi  - \nabla^E_{fX}\phi\\[3mm]
&=& \nnabla_{f\nabla_{X}}\phi  -f\nabla^E_{X}\phi\\[3mm]
&=&f\left(\nnabla_{\nabla_{X}}\phi  -\nabla^E_{X}\phi\right)\\[3mm]
&=&fK_0(X;\phi).
\end{array}
$$
Also, applying \eqref{eq-delta}:
$$
\begin{array}{rcl}
K_0(X;f\phi)&=&\nnabla_{\nabla_{X}}(f\phi)  - \nabla^E_{X}(f\phi)\\[3mm]
&=& \nabla_{X}(f)\phi+f\nnabla_{\nabla_{X}}\phi -  X(f)\phi -f\nabla^E_{X}\phi\\[3mm]
&=& X(f)\phi+f\nnabla_{\nabla_{X}}\phi -  X(f)\phi -f\nabla^E_{X}\phi\\[3mm]
&=&f\left(\nnabla_{\nabla_{X}}\phi  -\nabla^E_{X}\phi\right)\\[3mm]
&=&f K_0(X;\phi).
\end{array}
$$
\end{proof}
As an immediate consequence, we get the structure of the space of graded connections on the Cartan-Koszul supermanifold.
\begin{proposition}\label{prop4}
The space of connections on the graded vector bundle $\mathcal{L}=\om\otimes\ges$, over the supermanifold $(M,\Omega(M))$,
is an affine space modeled on:
$$
\Upgamma\left( T^\ast M \otimes \Uplambda T^\ast M\otimes {\rm End}(E)\right).
$$
\end{proposition}
The similarity between this space and the one modeling superconnections (see \eqref{modeled-Quillen}), will allow us to
associate a Quillen superconnection to any graded connection, in the next Section. For this, we will need some facts
about the covariant graded differential associated to a graded connection.

Given the supermanifold $(M,\om )$, and a graded vector bundle $\mathcal{L}=\om\otimes\ges$ over it, we can construct the
$\mathcal{L}-$valued superforms, $\oo((M,\om );\mathcal{L})$, and their derivations, as in Definition \ref{def2}.
Thus, if $\mathcal{D}\in\der\oo((M,\om );\mathcal{L})$, there exists a graded derivation, $\overline{\mathcal{D}} :\oo(M,\om)\to\oo(M,\om)$,
such that:
$$
[\mathcal{D},\lambda ]=\overline{\mathcal{D}}(\lambda) .
$$
As $\oo((M,\om );\mathcal{L})\simeq \oo(M,\om)\otimes \mathcal{L}$, any $\mathcal{D}\in\der\oo((M,\om );\mathcal{L})$ is
characterized by its associated derivation, $\overline{\mathcal{D}}$, and its action $\mathcal{D}(\phi)$, for $\phi\in\ges$.
Then, it is extended to all of $\oo((M,\om );\mathcal{L})$ as a derivation of $\Z\times\Z_2-$bidegree $(1,0)$.
\begin{definition}
Given a graded connection $\nnabla$ on the graded vector bundle $\mathcal{L}=\om\otimes\ges$ over $(M,\om )$,
its covariant graded differential is the graded derivation $\bm{\df}^{\nnabla}\in\der\oo((M,\om );\mathcal{L})$ given by
$$
\overline{\bm{\df}^{\nnabla}}=\bm{\df}\mbox{ and }\bm{\df}^{\nnabla}(\phi)=\nnabla \phi ,
$$
where $\bm{\df}$ is the graded differential.
\end{definition}
As a consequence of Proposition \ref{prop4}, every covariant
graded differential can be written locally as $\bm{\df}^{\nnabla}=\bm{\df}+P$, where $P\in\oo((M,\om);\mathrm{End}(\mathcal{L}))$ has
$\Z\times\Z_2-$bidegree $(1,0)$. Indeed, in view of the isomorphisms $\mathcal{L}\simeq\om\otimes\ges$, and
$\oo((M,\om);\mathrm{End}(\mathcal{L}))\simeq\oo(M,\om)\otimes\mathrm{End}(\mathcal{L})$,
we can consider $P\in\oo((M,\om);\mathrm{End}(E))$.
These computations can be readily generalized to any split supermanifold $(M,\se F))$, where $F\to M$ is a vector bundle over $M$.
Theorems \ref{thm1} and \ref{superbundles} can be applied in this case, and the following results, whose proof is obvious,
hold (we will apply them in Section \ref{curvatures} to the computation of the Chern classes of graded vector bundles over split supermanifolds).
\begin{proposition}
The space of connections on a graded vector bundle $\mathcal{L}$, over the supermanifold $(M,\se F))$, is an
affine superspace modeled on:
$$
\oo^{(1,+)}((M,\se F));\mathrm{End}(E))=\bigoplus_{k\in\mathbb{N}\bigcup\{0\}}\oo^{(1,2k)}((M,\se F));\mathrm{End}(E)),
$$
where $E=E_0\oplus E_1$ is a supervector bundle over $M$ such that $\mathcal{L}\simeq \se F)\otimes\ges$.
\end{proposition}
In particular, any graded connection $\nnabla$ on $\mathcal{L}$ over $(M,\se F))$,
induces a graded covariant differential $\bm{\df}^{\nnabla}$ which is an operator on $\oo((M,\om );E)$ of
$\Z\times\Z_2-$bidegree $(1,0)$, and any such operator can be locally written as
\begin{equation}\label{decomp}
\bm{\df}^{\nnabla}=\bm{\df}+P,\mbox{ for some }P\in\oo^{(1,+)}((M,\se F));\mathrm{End}(E)).
\end{equation}
\begin{proposition}\label{oolin}
The $\Z\times\Z_2-$bigraded algebra of $\mathrm{End}(E)-$valued superforms, $\oo((M,\se F));\mathrm{End}(E))$, is
isomorphic to the $\Z\times\Z_2-$bigraded algebra of operators $\mathcal{F}\in\mathrm{End}\oo((M,\se F) );E)$ which are $\oo(M,\se F))-$linear,
that is:
$$
\mathcal{F}(\mu(\lambda \otimes \phi ))=(-1)^{\mu\mathcal{F}}\mu\otimes\mathcal{F}(\lambda\otimes\phi),
$$
for every $\lambda,\mu\in\oo(M,\se F))$, and $\phi\in\ges$.
\end{proposition}
\begin{example}\label{sexample}
Given a graded connection $\nnabla$ on $\mathcal{L}=\se F)\otimes\ges$ over $(M,\se F))$, if $Q\in\oo((M,\se F));\mathrm{End}(E))$, then
$[\bm{\df}^{\nnabla} ,Q]$ is $\oo(M,\se F))-$linear as a graded operator on $\oo((M,\se F));E)$. Indeed, if we take any elements
$\lambda\in\oo(M,\se F))$, $s\in\oo((M,\se F));E)$, then it is clear that:
$$
Q(\lambda s)=(-1)^{\lambda Q}\lambda Q(s).
$$
Therefore (noticing that the $\Z\times\Z_2-$bidegree of $\bm{\df}^{\nnabla}$ is $(1,0)$, the same as that of $\bm{\df}$):
\begin{align*}
[\bm{\df}^{\nnabla} ,Q](\lambda s)&=\bm{\df}^{\nnabla}(Q(\lambda s))-(-1)^{Q\bm{\df}}Q(\bm{\df}^{\nnabla}(\lambda s)) \\
&=(-1)^{\lambda Q}\bm{\df}^{\nnabla}(\lambda Q(s))
-(-1)^{Q\bm{\df}}Q\left( (\bm{\df}\lambda )s+(-1)^{\lambda\bm{\df}}\lambda\bm{\df}^{\nnabla}s\right) \\
&=\lambda\left( (-1)^{\lambda(Q+\bm{\df})}\bm{\df}^{\nnabla}(Q(s))-(-1)^{Q\bm{\df}+\lambda(Q+\bm{\df})}Q(\bm{\df}^{\nnabla}s)\right) \\
&+(-1)^{\lambda Q}\bm{\df}\lambda\, Q(s)-(-1)^{\lambda Q}\bm{\df}\lambda\, Q(s) \\
&=(-1)^{\lambda (Q+\bm{\df})}\lambda [\bm{\df}^{\nnabla} ,Q](s).
\end{align*}
\end{example}

\section{Quillen superconnections induced by graded connections}\label{sec6}
Let $E=E_0\oplus E_1\to M$ be a supervector bundle on $M$, and consider the graded vector bundle $\mathcal{L}=\om\otimes\ges$ over
the supermanifold $(M,\om )$. The main results of this paper, are based on the following remarks.
\begin{enumerate}
\item Differential forms on $M$ with values on $E$ can be identified with $\mathcal{L}-$valued graded
differential $0-$forms on $(M,\Omega(M))$, $\oo^0((M,\om);\mathcal{L})$, through (see Example \ref{vvsforms}):
$$
\fvv\simeq\om\otimes\ges= \mathcal{L}\simeq \oo^0((M,\om);\mathcal{L}) .
$$
Explicitly, we will write $\al\otimes\phi\mapsto \al\phi$, where $\al\in\om$ and $\phi\in\ges$.
\item If $\nnabla$ is a graded connection on the graded vector bundle $\mathcal{L}=\om\otimes\ges$, over the
supermanifold $(M,\om)$, the covariant graded differential associated to $\nnabla$ acts on $E-$valued $0-$superforms:
$$\bm{\df}^{\nnabla}:\oo^0((M,\om);E)\to \oo^1((M,\om);E)$$
Explicitly, if $\alpha \phi\in \oo^0((M,\om);E)$, then  $\bm{\df}^{\nnabla}(\alpha\phi)$ is defined by
$$\langle \delta;\bm{\df}^{\nnabla}(\alpha\phi)\rangle = \nnabla_{\delta}(\alpha\phi)=
\delta(\alpha)\phi+(-1)^{|\delta||\alpha|}\alpha \nnabla_{\delta}\phi\in \Omega^0((M,\Omega(M));E),$$
for any $\delta\in \der\om$.
\item The usual exterior differential on $M$, $\df$, is a derivation of $\Omega(M)$.
Thus (as in Example \ref{ex3}), we can define the insertion (or evaluation) operator associated to it:
$$\iota_{{\rm d}}: \Omega^1((M,\Omega(M));E)\to  \Omega^0((M,\Omega(M));E).$$
\end{enumerate}
\begin{definition}
The composition of the previous three maps define an operator
$$D^{\nnabla}:\Omega(M;E)\to \Omega(M;E),$$ given by:
$$D^{\nnabla}(\alpha\otimes\phi) = \nnabla_{\rm d}(\alpha\phi)\in \Omega^0((M,\Omega(M));E) \simeq \Omega(M;E).$$
\end{definition}
\begin{theorem}
The operator so defined, $D^{\nnabla}$, is a Quillen superconnection.
\end{theorem}
\begin{proof}
First notice that, since the graded connection is an even map, and the $\Z-$degree of the exterior derivative is $1$,
the $\Z_2 -$degree of $\nnabla_{\rm d}(\alpha\phi)$ is
$$\left|\nnabla_{\rm d}(\alpha\phi)\right|= 1+|\alpha\phi|.$$
Therefore $D^{\nnabla}$ is of odd total degree.
Moreover, $D^{\nnabla}$ is a Quillen superconnection, because
$$\begin{array}{rcl}
D^{\nnabla}(\alpha\wedge (\beta\otimes\phi)) &=& \nnabla_{\rm d}(\alpha\wedge (\beta\phi ))\\[3mm]
 &=&{\rm d}\alpha\wedge (\beta\phi )+(-1)^{|\alpha|}\alpha\wedge\nnabla_{\rm d}(\beta\phi)\\[3mm]
 &=&{\rm d}\alpha\wedge (\beta\phi )+(-1)^{|\alpha|}\alpha\wedge D^{\nnabla}(\beta\otimes\phi),
\end{array}
$$
and this is \eqref{def-Quillen} in Definition \ref{def3}.
\end{proof}

Not all Quillen superconnection can be generated in this way, but we can prove the following
result.
\begin{theorem}
Any Quillen superconnection, $D$, can be written as
$$D = D^{\nnabla}+ N,$$
where
\begin{enumerate}[(1)]
\item $\nnabla$ is a graded connection on the graded vector bundle $\mathcal{L}=\om\otimes\ges$ on the supermanifold $(M,\om)$,
\item $N\in{\rm End}^1(E)$.
\end{enumerate}
Moreover, the decomposition is unique in the following sense: if $\nnabla'$ is another connection and $N'$ another tensor such that:
$$D = D^{\nnabla'}+ N',$$
then, $D^{\nnabla}= D^{\nnabla'}$, and $N=N'$.
\end{theorem}
\begin{proof}
Let us recall (see Example \ref{lieder}) that in a local coordinate system $\{ x^k\}_{1\leq k\leq \dim M}$ on $M$,
$$\begin{array}{rcl}
{\rm d}&=& {\mathscr L}_{\rm Id} = dx^k\wedge{\mathscr L}_{\partial_k}\\[3mm]
&=&
dx^k\wedge\left( \nabla_{\partial_k} +i_{\nabla \partial_k+T(\partial_k,\cdot)}\right) .
\end{array}$$
Therefore, according to Proposition \ref{tensor-associated-graded-connection}, for $\phi\in\Gamma(E)$:
$$\begin{array}{rcl}
\nnabla_{\rm d}\phi&=&
dx^k\wedge\left(\nnabla_{\nabla_{\partial_k}}\phi+\nnabla_{i_{\nabla \partial_k+T(\partial_k,\cdot)}}\phi\right)\\[4mm]
&=&
dx^k\wedge\left(\nabla^E_{\partial_k}\phi+K_0(\partial_k;\phi)+K_1({\nabla \partial_k +T(\partial_k,\cdot)};\phi)\right)\\[4mm]
&=&
\df^{\nabla^E}\phi+K_0^a(\ ;\phi)+\widetilde{K_1}(\ ;\phi),
\end{array}$$
where $K_0^a$ denotes the antisymmetrization of $K_0$, and $\widetilde{K_1}$ is given by the expression
$\widetilde{K_1}(\cdot ;\phi) = dx^k\wedge K_1(\nabla \partial_k +T(\partial_k,\cdot);\phi)$.
Let us check that they are tensor fields. If $\{y^\ell\}_{1\leq \ell\leq \dim M}$ is another set of
coordinate functions, then, by the change of coordinates formulae
$$
\partial_k = \frac{\partial y^\ell}{\partial x^k}\ \partial_{y^\ell},\qquad dx^k = \frac{\partial x^k}{\partial y^\ell}\ dy^\ell ,
$$
we get:
$$\begin{array}{rcl}
dx^k\wedge K_0(\partial_k;\phi)&=&
\frac{\partial x^k}{\partial y^\ell}\ dy^\ell\wedge \frac{\partial y^m}{\partial x^k}\ K_0(\partial_{y^m};\phi)\\[4mm]
&=&
\frac{\partial x^k}{\partial y^\ell}\ \frac{\partial y^m}{\partial x^k}\ dy^\ell\wedge K_0(\partial_{y^m};\phi)\\[4mm]
&=&
\delta^m_\ell\ dy^\ell\wedge K_0(\partial_{y^m};\phi)\\[4mm]
&=&
dy^\ell\wedge K_0(\partial_{y^\ell};\phi).
\end{array}
$$
The computation needed for proving that $\widetilde{K_1}$ is also a tensor field, is longer but
equally straightforward.\\
Now, we can readily check that the operations $K_0\mapsto K^a_0$ and $K_1\mapsto \widetilde{K_1}$
are surjective. For instance, let $T\in\se^{2k+3}T^*M)$ be given, and let us write it locally as:
$$
T=T_{i_1 \cdots i_{2k+3}}\df x^{i_1}\wedge \cdots\wedge \df x^{i_{2k+3}}.
$$
To find a $K\in\Upgamma (T^*M\otimes\Uplambda^{2k+1}T^*M)$ such that $T=\df x^p\wedge K(\nabla \partial_p +T(\partial_p,\cdot);\phi)$,
we also express it locally, as $K=K_{j_1 \cdots j_{2k+1}}\df x^{j_1}\wedge \cdots \wedge \df x^{j_{2k+1}}\otimes a_k \df x^k$, and
write a system of equations for the unknowns $K_{j_1 \cdots j_{2k+1}},a_k, \Upgamma^r_{pq}$, where $\Upgamma^r_{pq}$ are the
Christoffel symbols of the linear connection $\nabla$ on $TM$. To this end, note that:
$$
\nabla \partial_p +T(\partial_p,\cdot)=\Upgamma^r_{qp}\partial_r\otimes\df x^q+2(\Upgamma^r_{pq}-\Upgamma^r_{qp})\partial_r\otimes\df x^q
=(2\Upgamma^r_{pq}-\Upgamma^r_{qp})\partial_r\otimes\df x^q .
$$
Call $B^r_{pq}=2\Upgamma^r_{pq}-\Upgamma^r_{qp}$. Then (denoting by $[ab]$ the antisymmetrization in the indexes $a,b$):
\begin{align*}
\df x^p\wedge K(\nabla \partial_p +T(\partial_p,\cdot);\phi)&=
K_{j_1 \cdots j_{2k+1}}a_k B^k_{[pq]}\df x^p \wedge \df x^{j_1}\wedge \cdots \wedge \df x^{j_{2k+1}}\wedge \df x^q , \\
&=K_{j_1 \cdots j_{2k+1}}a_k B^k_{[pq]}\df x^{j_1}\wedge \cdots \wedge \df x^{j_{2k+1}}\wedge\df x^q \wedge \df x^p .
\end{align*}
Thus, we are led to the systems of equations
$$
T_{i_1 \cdots i_{2k+3}}=K_{i_1 \cdots i_{2k+1}}a_k B^k_{[i_{2k+2}i_{2k+3}]} ,
$$
which clearly admits a solution.\\
On the other hand, if:
$$\begin{array}{l}
K_0\in\Upgamma(T^*M\otimes \Uplambda^{2k}T^\ast M\otimes{\rm End}^0(E)),\mbox{ then } K_0^a\in\Upgamma(\Uplambda^{2k+1}T^\ast M\otimes{\rm End}^0(E)),\\[3mm]
K_0\in\Upgamma(T^*M\otimes \Uplambda^{2k+1}T^\ast M\otimes{\rm End}^1(E)),\mbox{ then } K_0^a\in\Upgamma(\Uplambda^{2k+2}T^\ast M\otimes{\rm End}^1(E)),\\[3mm]
K_1\in\Upgamma(T^*M\otimes \Uplambda^{2k}T^\ast M\otimes{\rm End}^1(E)),\mbox{ then } \widetilde{K_1}\in\Upgamma(\Uplambda^{2k+2}T^\ast M\otimes{\rm End}^1(E)),\\[3mm]
K_1\in\Upgamma(T^*M\otimes \Uplambda^{2k+1}T^\ast M\otimes{\rm End}^0(E)),\mbox{ then } \widetilde{K_1}\in\Upgamma(\Uplambda^{2k+3}T^\ast M\otimes{\rm End}^0(E)).
\end{array}
$$
So, by comparing these tensor fields with those that model Quillen superconnections (see \eqref{modeled-Quillen}),
we see that all of them can be obtained this way, except one. The term that never appears is
$\Upgamma(\Uplambda^{0}T^\ast M\otimes{\rm End}^1(E)) = \Upgamma({\rm End}^1(E))$.
This is why we need to add an $N\in{\rm End}^1(E)$ to get an arbitrary Quillen superconnection.
\end{proof}
\begin{example}\label{quillen-ex}
Quillen's paper \cite{Qui} deals mainly with the particular case of superconnections $D$ (on the superbundle
$E=E_0\oplus E_1\to M$) of total degree not only odd, but exactly $1$.
As seen in Example \ref{example-quillen}, to characterize one of these superconnections we need a connection $\nabla^E$ on the vector
bundle $E\to M$, and an odd vector-valued form $P=A+L\in\oo^0(M;\mathrm{End}^1(E))\oplus \oo^1(M;\mathrm{End}^0(E))$. In view of the
theorem above, all of these superconnections are induced by graded connections on the graded vector bundle $\mathcal{L}=\om\otimes\Upgamma (E)$,
over the supermanifold $(M,\om )$. It suffices to take $N=A$, $K_0=L$, and $K_1=0$.
\end{example}

The next result, which is an immediate corollary of the preceding Theorem, tells us when two graded
connections induce the same Quillen superconnections.
\begin{proposition}
$D^{\nnabla}= D^{\nnabla'}$ if and only if
\begin{enumerate}
\item $(K_0-K'_0)^a = 0$ and,
\item $\widetilde{K_1-K'_1} = 0$.
\end{enumerate}
\end{proposition}

There is a nice relation between the curvature of a graded connection and that of its associated Quillen
superconnection, involving the exterior differential of $M$, which is expressed in the following result.
\begin{theorem}
For any graded connection, $\nnabla$:
$$
R^{D^{\nnabla}}=\frac{1}{2}R^{\nnabla}(\df,\df ).
$$
\end{theorem}
\begin{proof}
By definition, for any $\phi\in\oo^0((M,\om );E)\simeq \Upgamma (E)$,
$$
R^{\nnabla}({\rm d},{\rm d})\phi= \nnabla_{{\rm d}}\nnabla_{{\rm d}}\phi-(-1)^{|{\rm d}||{\rm d}|}\nnabla_{{\rm d}}\nnabla_{{\rm d}}\phi-\nnabla_{[{\rm d},{\rm d}]}\phi ,
$$
but $|\df |=1$ and $[\df ,\df ]=2\df^2=0$, so this gives:
$$
R^{\nnabla}({\rm d},{\rm d})\phi = 2\nnabla_{{\rm d}}^2\phi .
$$
The extension to all of $\oo^k((M,\om );E)$ is then immediate.
\end{proof}
A straightforward consequence is the following.
\begin{corollary}
If $D=D^{\nnabla}+N$, then:
$$
R^D = \frac12 R^{\nnabla}({\rm d},{\rm d})+\nnabla_{{\rm d}}N+ N^2.
$$
\end{corollary}
\begin{example}
Returning to the case presented in Example \ref{quillen-ex}, for a \emph{complex} superbundle
$E=E_0\oplus E_1\to M$, if $D$ is a Quillen superconnection given by
$D=\df^{\nabla^E}+A+L$, we have seen that it can be expressed as $D=D^{\nnabla} +A$, where $\nnabla$ is defined by:
\begin{align*}
\nnabla_{\nabla_X}\phi &=\nabla^E_X \phi +L(X;\phi )\\
\nnabla_{i_X} \phi &=0,
\end{align*}
for any $\phi\in \se)\simeq \Omega^0((M,\Omega(M));E)$, with $L\in\Omega^1 (M;\mathrm{End}^0 (E))$.
By the preceding corollary, $R^D =\frac{1}{2}R^{\nnabla} (\df ,\df )+\nnabla_\df A+A^2$.
Moreover, in the particular case of $A=0$, and $L$ of the form:
$$
L=i
\begin{pmatrix}
0 &  u^* \\
u &  0   \\
\end{pmatrix}
,
$$
with $u:E_0\to E_1$, and $u^*$ its adjoint with respect to some metric, we have $D=D^{\nnabla}$, and so $R^D=\frac{1}{2}R^{\nnabla} (\df,\df)$.
The complex scalar field $u$ is identified, in Physics, with the tachyon field condensing between a brane-antibrane system.
\end{example}

\section{Chern classes of graded vector bundles}\label{curvatures}

Throughout this Section, $(M,\se F))$ will be a split supermanifold, and $\mathcal{L}=\se F)\otimes \ges$ a graded vector bundle over it,
with $E=E_0\oplus E_1\to M$ a supervector bundle over $M$.

We want to associate to any graded connection $\nnabla$ on $\mathcal{L}$, a set of closed superforms.
As in the case of Quillen superconnections, this can be done with the aid of the supertrace: if $A$ is any
$G-$graded $\K-$algebra, a supertrace on $A$ is any $\K-$linear form on $A$ such that it vanishes on the supercommutator \eqref{scommutator}.
The supertrace is not unique, in general, but for endomorphisms of a supervector bundle, $A=\mathrm{End}(E)$, a canonical
supertrace $\str$ can be defined as follows: any $T\in\mathrm{End}(E)$ decomposes as
$T=T_0\oplus T_1\in\mathrm{End}^0(E)\oplus\mathrm{End}^1(E)$, with a further decomposition
$T_0=T^0_0\oplus T^1_0\in(E^*_0\otimes E_0)\oplus(E^*_1\otimes E_1)$. Then:
$$
\str (T)=
\begin{cases}
\mathrm{Tr}(T^0_0)-\mathrm{Tr}(T^1_0) \mbox{, if }T=T_0 \\[4mm]
0 \mbox{, if }T=T_1
\end{cases}.
$$
This linear form on the fibers of the supervector bundle $\mathrm{End}(E)$, can be seen as a $\oo^0(M,\se F))-$module morphism
$\oo^0((M,\se F));\mathrm{End}(E))\to\oo^0(M,\se F))$, which can be extended to a $\oo(M,\se F))-$module morphism:
$$
\str:\oo((M,\se F));\mathrm{End}(E))\to \oo(M,\se F)) .
$$
That property of $\oo(M,\se F))-$linearity, guarantees that we can take the supertrace of $[\bm{\df}^{\nnabla},Q]$ (recall Example \ref{sexample}).
\begin{proposition}\label{ellema}
For $\nnabla$ any graded connection on $\mathcal{L}=\se F)\otimes \ges$, and $Q\in\oo^k((M,\se F));\mathrm{End}(E))$, the
following holds:
$$
\str [\bm{\df}^{\nnabla},Q]=\bm{\df}\str(Q) .
$$
\end{proposition}
\begin{proof}
Working locally, by the decomposition \eqref{decomp}, we can write $\bm{\df}^{\nnabla}=\bm{\df}+P$. Put also
$Q=\alpha\, T$, where $\alpha\in\oo^k (M,\Upgamma (\Uplambda F))$ and $T\in\Upgamma (\mathrm{End}E)$. Then, as $\str$ vanishes on supercommutators:
$$
\str [\bm{\df}^{\nnabla},Q]=\str [\bm{\df}+P,\alpha\, T]=\str [\bm{\df},\alpha\, T]+\str [P,Q]=\str (\bm{\df}\alpha\, T).
$$
Finally, again because of the $\oo(M,\se F))-$linearity, $\bm{\df}$ commutes with
$\str$, giving the statement.
\end{proof}
\begin{definition}
Let $\nnabla$ be a graded connection on the graded vector bundle $\mathcal{L}=\se F)\otimes \ges$ over the supermanifold
$(M,\se F))$. For any $k\in\N$, the characteristic superform (or Chern superform) of $\nnabla$, of degre $2k$, is the
superform on $(M,\se F))$ given by $\str \left(\bm{\df}^{\nnabla^{2k}}\right)$. For the sake of simplicity, sometimes it will be
denoted by $\str (\nnabla^{2k})$.
\end{definition}

The next result is proved following the ideas presented in \cite{Qui}, but with some modifications to take into account that
we are now dealing with vector-valued superforms (not just forms), and that the natural grading in this context is given by the
$\Z\times\Z_2-$degree, not the total degree.
\begin{theorem}
Let $\nnabla$ be a graded connection on the graded vector bundle $\mathcal{L}=\se F)\otimes \ges$ over $(M,\se F))$, and let $k\in\N$.
Then:
\begin{enumerate}[(1)]
\item The characteristic forms $\str (\nnabla^{2k})$ are closed even superforms.
\item If $\nnabla_0$ and $\nnabla_1$ are two graded connections on $\mathcal{L}$, then the superforms $\str (\nnabla^{2k}_0)$
and $\str (\nnabla^{2k}_1)$ lie in the same graded De Rham cohomology class.
\end{enumerate}
\end{theorem}
\begin{proof}
\mbox{}
\begin{enumerate}[(1)]
\item By applying proposition \ref{ellema}, we get:
$$
\bm{\df}\str (\nnabla^{2k})=\str \left[ \bm{\df}^{\nnabla},\bm{\df}^{\nnabla^{2k}}\right]
=\str \left[ \bm{\df}^{\nnabla},\left( \bm{\df}^{\nnabla^2}\right)^k\right] =0,
$$
because $\bm{\df}^{\nnabla^{2k}}$ has bidegree $(2k,0)$, so
\begin{equation}\label{bianchi}
\left[ \bm{\df}^{\nnabla},\bm{\df}^{\nnabla^{2k}}\right] =\left( \bm{\df}^{\nnabla}\right)^{2k+1}-\left( \bm{\df}^{\nnabla}\right)^{2k+1}=0
\end{equation}
(this can be thought of as a form of Bianchi identity).
\item Given $\nnabla_0$ and $\nnabla_1$, define for $t\in [0,1]$ the family of graded connections $\nnabla_t =(1-t)\nnabla_0 +t\nnabla_1$. Then:
$$
\bm{\df}^{\nnabla_t}=(1-t)\bm{\df}^{\nnabla_0} +t\bm{\df}^{\nnabla_1} =\bm{\df}^{\nnabla_0} +tP,
$$
with $P\in\oo^{(1,+)}((M,\se F));\mathrm{End}(E))$. It is immediate that
\begin{align*}
\frac{\df}{\df t}\left( \bm{\df}^{\nnabla_t}\right)^2
&=\frac{\df}{\df t}\left( \frac{1}{2}\left[\bm{\df}^{\nnabla_t},\bm{\df}^{\nnabla_t}\right]\right) \\
&=\frac{\df}{\df t}\left( \frac{1}{2}\left[\bm{\df}^{\nnabla_0} +tP,\bm{\df}^{\nnabla_0} +tP\right] \right)\\
&=\left[\bm{\df}^{\nnabla_0} +tP,P\right] \\
&=\left[\bm{\df}^{ \nnabla_t},\frac{\df}{\df t}\bm{\df}^{\nnabla_t} \right],
\end{align*}
so, substituting this identity, using \eqref{bianchi} for $\bm{\df}^{\nnabla_t}$, and Proposition \ref{ellema}:
\begin{align*}
\frac{\df}{\df t}\str\left( \bm{\df}^{\nnabla^{2}_t}\right)^k
&=\str\left( k \left( \bm{\df}^{\nnabla^{2}_t}\right)^{k-1} \frac{\df}{\df t} \bm{\df}^{\nnabla^2_t} \right) \\
&=\str\left( k\left[ \bm{\df}^{\nnabla_t},\left( \bm{\df}^{\nnabla^2_t}\right)^{k-1}\frac{\df}{\df t} \bm{\df}^{\nnabla_t}\right] \right)\\
&=\bm{\df}\str \left( k \left( \bm{\df}^{\nnabla^2_t}\right)^{k-1} \frac{\df}{\df t} \bm{\df}^{\nnabla_t} \right)\\
&=\bm{\df}\str \left( k \left( \bm{\df}^{\nnabla^2_t}\right)^{k-1} P \right).
\end{align*}
Finally, integrating on the interval $[0,1]$, we get:
$$
\str\left( \bm{\df}^{\nnabla_1}\right)^{2k}-\str\left( \bm{\df}^{\nnabla_0}\right)^{2k}=
\bm{\df}\int^1_0\str \left( k\left( \bm{\df}^{\nnabla^2_t}\right)^{k-1}P\right)\df\, t.
$$
Thus, $\str( \nnabla_0^{2k})$ and $\str( \nnabla_1^{2k})$ lie in the same graded De Rham cohomology class.
\end{enumerate}
\end{proof}
\begin{definition}
For any $k\in\N$, the Chern class of degree $2k$ of the graded vector bundle $\mathcal{L}$ over the supermanifold $(M,\se F))$, denoted
$\mathrm{ch}_k(\mathcal{L})$, is the cohomology class of any superform $\str (\nnabla^{2k})$ (with $\nnabla$ any
graded connection on $\mathcal{L}$).
\end{definition}

On any supermanifold $(M,\mathcal{A})$, the structural morphism $\mathcal{A}\stackrel{\sim}{\rightarrow}\cinfm$ gives rise to
another morphism $\kappa :\Omega (M,\mathcal{A})\to \Omega (M)$, which, in turn, induces an \emph{isomorphism}
between the De Rham cohomologies of $(M,\mathcal{A})$ and $M$ (see \cite{Kos}, sec. 4.6). In the case of a split supermanifold
$(M,\se F))$, the structural morphism $\sim$ is just the projection onto the $\mathbb{Z}-$degree $0$, so
the computations can be easily done.
\begin{theorem}\label{lastthm}
The Chern classes of the supervector bundle $E=E_0\oplus E_1\to M$, over the manifold $M$, and those of the graded vector
bundle $\mathcal{L}=\se F)\otimes \ges$, over the supermanifold $(M,\se F))$, are the same.
\end{theorem}
\begin{proof}
Consider $\nabla^E$, a degree-preserving connection on $E$, and $\nabla$ a linear connection on $TM$.
As the Chern classes are independent of the connection, we will choose the superconnection
$D=\mathrm{d}^{\nabla^E}$ to compute those of $E$, and the graded connection given by
$$
\begin{cases}
\nnabla_{\nabla_X}\phi = \nabla^E_X \phi \\
\nnabla_{i_X}\phi =0
\end{cases},
$$
to compute those of $\mathcal{L}$. Then, it is clear that
$$
\kappa (\bm{\df}^{\nnabla} \phi)=\nabla^E \phi =\df^{\nabla^E}\phi,
$$
so the statement follows.
\end{proof}
The Chern classes of some graded vector bundles of interest can be computed directly from this result.
\begin{corollary}
The Chern classes of the graded vector bundle $\mathcal{L}=\se F)\otimes \ges$ over $(M,\se F))$, such that $E=E_0\oplus E_1\to M$ is
a supervector bundle, are:
$$
\mathrm{ch}_k(\mathcal{L})=\mathrm{ch}_k(E_0)-\mathrm{ch}_k(E_1).
$$
\end{corollary}
\begin{proof}
The statement is true for Quillen superconnections (see \cite{Qui}). So it is true for graded vector bundles, by Theorem \ref{lastthm}.
\end{proof}
Recall that the supertangent bundle of a split supermanifold $(M,\se F))$ is given by the sheaves of derivations
$\der(\se F))$, and that we have the isomorphism $\der\se F)\simeq \se F) \otimes \Gamma(TM\oplus F^*)$ (see Examples \ref{stangent}
and \ref{taniso}). The supercotangent structural sheaf is then isomorphic to the dual sheaf
$\der^* \se F)\simeq \se F) \otimes \Gamma(T^* M\oplus F)$. This, together with the fact that usual Chern classes
satisfy the duality property $\mathrm{ch}_k(F^*)=(-1)^k\, \mathrm{ch}_k(F)$ for any vector bundle $F$, imply an analogous result
for the supertangent and supercotangent bundles.
\begin{corollary}
For any split supermanifold $(M,\se F))$, the Chern classes of the supertangent bundle satisfy
$$
\mathrm{ch}_k(\der(\se F)))=\mathrm{ch}_k(TM)-\mathrm{ch}_k(F^*),
$$
and those of the cotangent superbundle
$$
\mathrm{ch}_k(\der^* (\se F)))=\mathrm{ch}_k(T^* M)-\mathrm{ch}_k(F)=(-1)^k\, \mathrm{ch}_k(\der(\se F))).
$$
\end{corollary}
\begin{example}
The Chern classes of both, the supertangent and the supercotangent bundle of the Cartan-Koszul supermanifold, all vanish.
\end{example}
In the non-graded case, Chern classes of degree greater than the dimension of the base manifold vanish, as a consequence
of the existence of a top degree in the De Rham cohomology, given precisely by $\dim M$. For supermanifolds, it is a well-known
fact that this top degree does not exist (see \cite{Kos}) so, \emph{a priori}, they do not necessarily vanish. However, it is a
direct consequence of Theorem \ref{lastthm} that the classical result is still valid in the super context.
\begin{corollary}
The Chern classes of degree $2k$ of the graded vector bundle $\mathcal{L}=\se F)\otimes \ges$ over $(M,\se F))$, such that
$2k>\dim M$, are zero.
\end{corollary}

\noindent \textbf{Acknowledgements.}
Research partly supported by a CONACyT-M\'exico grant, Project CB-2012 179115 (JAV), and a grant from the Spanish
Ministry of Science and Innovation, Project MTM2012-33073 (JM).


\begin{thebibliography}{99}

\bibitem{MHY} M. Alishahiha, H. Ita, Y. Oz, {\it On superconnections and the tachyon effective action}.
Phys. Lett. B \textbf{503} (2001) 181--188.

\bibitem{BBHP} C. Bartocci, U. Bruzzo, D. Hern\'andez-Ruip\'erez and V.G. Pestov,
{\it Foundations of supermanifold theory: the axiomatic approach}.
Diff. Geom. and its Appl., \textbf{3} 2 (1993) 135--155.
%

\bibitem{BGV} N. Berline, E. Getzler and M. Vergne, {\it Heat kernels and Dirac operators},
Springer-Verlag, Grundleheren der mathematische Wissenchaften, {\bf 298} (1992).
%

\bibitem{Sardanas} G. Giachetta, L. Mangiarotti and G. Sardanashvily,
{\it Geometric and Algebraic Topological Methods in Quantum Mechanics},
World Scientific (2005).

\bibitem{KM} Y. Kosmann-Schwarzbach, J. Monterde, {\it Divergence
operators and odd Poisson brackets}. Ann. Inst. Fourier, {\bf 52}
(2002) 419--456.

\bibitem{Kos} B. Kostant, {\it Graded manifolds, graded Lie theory, and prequantization}.
Differential Geometrical Methods in Mathematical Physics,
Lecture Notes in Mathematics \textbf{570} (1977) 177--306.


\bibitem{Kz2} J. L. Koszul, {\it Connections and splittings of supermanifolds}.
Differential Geom. and its Appl., {\bf 4} (1994) 151--161.
%

\bibitem{KL} P. Kraus and F. Larsen, {\it Boundary string field theory of the $D\overline{D}$ system}.
Phys. Rev. \textbf{D 63} (2001) 106004, 17 pp.

\bibitem{Mic} P. W. Michor, {\it Remarks on the Fr\"olicher-Nijenhuis bracket}.
Proc. Conference Differential Geometry and its Applications, Brno 1986, D. Reidel, (1987), 197--220.

\bibitem{MM} J. Monterde, A. Montesinos,
{\it Integral curves of derivations}. Ann. Global Anal. Geom., {\bf 6} (1988) 177--189.
%

\bibitem{JA} J. Mu\~noz-Masqu\'e and O. A. S\'anchez-Valenzuela,
{\it Natural quotients on split supercotangent bundles and their canonical supersymplectic structures}.
Diff. Geom. and its Appl., \textbf{12} 1 (2000) 85--103.

\bibitem{Lee1} C. Y. Lee and Y. Ne'eman, {\it Superconnections and electroweak symmetry breaking},
Phys. Lett. \textbf{B264} (1991) 389--395.

\bibitem{Lee2} C. Y. Lee, D. S. Hwang and Y. Ne'eman, {\it BRST quantization of gauge theory in noncommutative geometry: Matrix derivative approach},
    J. Math. Phys. \textbf{37} (1996) 3725--3738.

\bibitem{NS1} Y. Ne'eman and S. Sternberg, {\it Superconnections and internal supersymmetry dynamics}.
Proc. Natl. Acad. Sci. USA., \textbf{87} (1990) 7875--7877.

\bibitem{NS2} Y. Ne'eman and S. Sternberg, {\it Internal supesymmetry and superconnections}.
Symplectic Geometry and Mathematical Physics. Birkh$\ddot{a}$user Prog. in Math., Vol. 99 (1991).

\bibitem{OS} K. Olsen and R. J. Szabo, {\it Constructing D-branes from K-theory}, Adv. Theoret. Math. Phys. \textbf{3} (1999) 889--1025.

\bibitem{Qui} D. Quillen, {\it Superconnections and the Chern character}. Topology {\bf 24} (1985) 89--95.

\bibitem{Roe} G. Roepstorff, {\it Superconnections and the Higgs field},
J. Math. Phys. \textbf{40} (1999) 2698--2715.

\bibitem{Rot} M. J. Rothstein, {\it The Axioms of Supermanifolds and a New Structure Arising From Them}.
Trans. of the Am. Math. Soc., \textbf{297} 1 (1986) 159--180.
%

\bibitem{Roy} D. Roytenberg, {\it AKSZ-BV formalism and Courant algebroid-induced topological field theories},
Lett. in Math. Phys. \textbf{79} (2007) 143--159.

\bibitem{Sza} R. J. Szabo, {\it Superconnections, anomalies and non-BPS brane charges}.
J. of Geom. and Phys., \textbf{43} (2002) 241--292.

\bibitem{Tahi} M. Tahiri, {\it Geometry and off-shell nilpotency for four-dimensional BF theories},
Phys. Lett. \textbf{B325} (1994) 71--76.

\bibitem{Uhl} A. Uhlman, {\it The Cartan algebra of exterior differential forms as a supermanifold:
morphisms and manifolds associated with them}. J. of Geom. and Phys., \textbf{1} 3 (1984) 25--38.

\bibitem{Vain} Yu. Vaintrob, {\it Lie algebroids and homological vector fields},
Russ. Math. Surv. \textbf{52} (1997) 428--429.

\bibitem{Vor} T. Voronov, {\it Graded manifolds and Drinfeld doubles for Lie bialgebroids},
Contemp. Math. AMS \textbf{315} (2002).

\bibitem{Wit} E. Witten, {\it $D-$branes and $K-$theory}.
J. of High Energy Phys., \textbf{12} 019 (1998) 41 pp.

\bibitem{Yasui} Y. Yasui and T. Ootsuka, {\it Spin$(7)$ holonomy manifold and superconnection}, Class. Quantum Grav. \textbf{18} (2001) 807--816.

\end{thebibliography}
\end{document}